\documentclass[english]{article}
\usepackage[T1]{fontenc}
\usepackage[latin9]{inputenc}
\usepackage{mathrsfs}
\usepackage{amsmath}
\usepackage{amsthm}
\usepackage{amssymb}
\usepackage{graphicx}

\makeatletter
\newcommand{\lyxaddress}[1]{
\par {\raggedright #1
\vspace{1.4em}
\noindent\par}
}
\theoremstyle{plain}
\newtheorem{thm}{\protect\theoremname}[section]
  \theoremstyle{definition}
  \newtheorem{problem}[thm]{\protect\problemname}
  \theoremstyle{plain}
  \newtheorem{cor}[thm]{\protect\corollaryname}
  \theoremstyle{plain}
  \newtheorem{prop}[thm]{\protect\propositionname}

\usepackage{babel}

\makeatother

\usepackage{babel}
  \providecommand{\corollaryname}{Corollary}
  \providecommand{\problemname}{Problem}
  \providecommand{\propositionname}{Proposition}
\providecommand{\theoremname}{Theorem}

\begin{document}
\global\long\def\d#1{\,{\rm d}#1}

\global\long\def\R{\mathbb{R}}

\global\long\def\C{\mathbb{C}}

\global\long\def\Z{\mathbb{Z}}

\global\long\def\N{\mathbb{N}}

\global\long\def\Q{\mathbb{Q}}

\global\long\def\T{\mathbb{T}}

\global\long\def\F{\mathbb{F}}

\global\long\def\vp{\varphi}

\global\long\def\Sph{\mathbb{S}}

\global\long\def\sub{\subset}

\global\long\def\one{\mathbbm1}

\global\long\def\vol#1{\text{vol}\left(#1\right)}

\global\long\def\EE{\mathbb{E}}

\global\long\def\sp{{\rm sp}}

\global\long\def\iprod#1#2{\langle#1,\,#2\rangle}

\global\long\def\uball{B_{2}^{n}}

\global\long\def\conv#1{{\rm conv}\left(#1\right)}

\global\long\def\met#1{{\rm M}\left(#1\right)}

\global\long\def\mfloat#1#2{{\rm M}_{#1}\left(#2\right)}

\global\long\def\mwfloat#1#2#3{{\rm M}_{#1}\left(#2,#3\right)}

\global\long\def\supp#1{{\rm supp}\left(#1\right)}

\global\long\def\eps{\varepsilon}

\global\long\def\PP{\mathbb{P}}

\global\long\def\vein#1{{\rm vein}\left(#1\right)}

\global\long\def\fvein#1{{\rm vein^{*}}\left(#1\right)}

\global\long\def\bfvein#1{{\rm \mathbf{vein^{*}}}\left(\mathbf{#1}\right)}

\global\long\def\ill#1{{\rm ill\left(#1\right)}}

\global\long\def\fill#1{{\rm ill^{*}\left(#1\right)}}

\global\long\def\ovr#1{{\rm ovr\left(#1\right)}}

\global\long\def\norm#1{\left\Vert #1\right\Vert }

\global\long\def\proj{{\rm Pr}}

\global\long\def\ra{\Rightarrow}

\global\long\def\H#1#2{H^{+}\left(#1,\,#2\right)}

\title{\textsc{John's position is not good for approximation }}

\author{Han Huang }
\maketitle

\lyxaddress{Department of Mathematics, University of Michigan, Ann Arbor, MI
48109-1043.}
\begin{abstract}
A well-known consequence of John's theorem states that any symmetric
convex body $K\sub\R^{n}$ in John's position can be approximated
by a polytope $P$ with a polynomial number of facets in $n$, so
that $P\sub K\sub\sqrt{n}P$. This results extends to the non-symmetric
case if the homothety ratio grows to $n$. In this note, we study
how well this result holds in the non-symmetric case, if the homothety
ratio is reduced below $n$. We prove the following: For $R=o\left(n\right)$
and a sufficiently large $n$, there exists a convex body $K\subset\mathbb{R}^{n}$
in John's position for which there is no polytope $P$ with a polynomial
number of facets, such that $K\subset P\subset RK$. Moreover, for
$R=O\left(\sqrt{n}\right)$, there exists a convex body for which
a polytope with an exponential number of facets is needed. 
\end{abstract}

\section{Introduction}

One of the most natural questions in convex geometry is how well can
a convex body in $\R^{n}$ (i.e., compact, convex set with non-empty
interior) be approximated by polytopes with as few facets (or vertices)
as possible. How closely a polytope approximates a convex body can
be measured in different ways. In this paper, we are interested in
the Banach-Mazur distance which, for (origin) symmetric convex bodies
$K,L\sub\R^{n}$, is defined by 
\[
d_{BM}(K,L):=\inf\{r\ge1\,:\,\exists T\in GL_{n}(\mathbb{R})\text{ such that }K\subset TL\subset rK\}.
\]
Denote the Euclidean unit ball in $\R^{n}$ by $B_{2}^{n}$. The following
lower bound was proven independently (and using different methods)
in \cite{BaranyFuredi}, \cite{BourganLindenstraussMilman}, \cite{CarlPajor}
and \cite{Gluskin}. For any polytope $P$ with $m$ facets, one has
$d_{BM}(B_{2}^{n},P)\ge c\sqrt{\frac{n}{\log(\frac{m}{n})}}$. Recently,
Barvinok \cite{Barvinok} showed that for any symmetric convex body
$K\subset\mathbb{R}^{n}$, one can find a polytope $P$ with a number
of facets, $m$, which is at least polynomial in $n$, such that $d_{BM}(K,P)=O(\sqrt{\frac{n\log(n)}{\log(m)}})$.\\

For the non-symmetric case, one has to modify the definition of the
Banach-Mazur distance: For convex bodies $K,L\sub\R^{n}$, the Banach-Mazur
distance $d_{BM}(K,L)$ is defined by 
\[
d_{BM}(K,L):=\inf\left\{ \begin{array}{c}
r\ge1\,:\,\exists T\in GL_{n}(\mathbb{R})\text{ and }x,y\in\mathbb{R}^{n}\\
\text{ such that }K-x\subset T(L-y)\subset r(K-x)
\end{array}\right\} .
\]
The choice of the origin $x$ of $K$ in this definition is crucial.
For a symmetric convex body, classical choices such as the center
of mass, the center of the John ellipsoid, and the Santaló point,
all coincide with the center of symmetry. However, this is not the
case for a general convex body, a fact which makes the choice of origin
an obstacle.

The following results make use of the center of mass as the origin.
The first result, by Szarek \cite{Szarek}, states that for any convex
body in $\mathbb{R}^{n}$, there exists a polytope with either $m$
facets or $m$ vertices, such that $d_{BM}(K,P)\le\frac{n}{\log(\frac{m}{n})}$.
Using a random method, Brazitikos, Chasapis, and Hioni obtain an upper
bound of the order of $\frac{n}{\sqrt{\log(\frac{m}{n})}}$, where
$m$ is the number of vertices.

In this paper, we consider another natural choice of an origin. Namely,
we fix the origin of a convex body $K\sub\R^{n}$ as the center of
its John ellipsoid. The John ellipsoid associated to a convex body
$K\sub\R^{n}$ is defined as the (unique) maximal volume ellipsoid
contained in $K$. We say that a convex body is in John's position
if its maximal volume ellipsoid is the unit Euclidean ball $B_{2}^{n}$.
For any convex body $K$, there exists an affine transformation $T$
such that $TK$ is in John's position. If we consider convex bodies
in John's position, then the origin is the center of their John ellipsoids.
John's position has many consequences that have been useful for solving
many problems, see e.g., \cite{AGA} and the reference therein. For
example, the only known proof showing that a symmetric convex body
$K$ can be approximated by a polytope $P$ with polynomial number
of facets in $n$ so that $d_{BM}(K,P)=O(\sqrt{n})$ relies on John's
position. More precisely, the proof uses the existence of $m=O\left(n^{2}\right)$
contact points $\{x_{i}\}_{i=1}^{m}$ (see definition below) that
form an identity decomposition. The polytope is then defined by $P:=\{x\in\mathbb{R}^{n}\,:\,\left|\langle x,x_{i}\rangle\right|\le1\quad\forall i\le m\}$.
In particular, $P$ satisfies 
\[
B_{2}^{n}\subset K\subset P\subset\sqrt{n}B_{2}^{n}.
\]
Thus, $d_{BM}(P,K)\le\sqrt{n}$. We remark that Barvinok's result
\cite{Barvinok} also relies on contact points. In the non-symmetric
case, however, the above construction can only imply

\[
B_{2}^{n}\subset K\subset P\subset nB_{2}^{n},
\]
which shows that $d_{BM}(K,P)\le n$.\\
 \\
 In this paper we investigate the following problem. 
\begin{problem}
\label{prob:main}Let $R=o(n)$, and $K\subset\mathbb{R}^{n}$ be
a convex body in John's position. Is there a polytope $P$ whose number
of facets is polynomial in $n$, such that 
\[
K\subset P\subset RK.
\]
\end{problem}

We prove the following main theorem. 
\begin{thm}
\label{thm:MAIN}For a sufficiently large $n$ and for any $c_{0}\sqrt{n}\le R\le c_{1}n$,
there exists a convex body $K\subset\mathbb{R}^{n}$ whose John's
ellipsoid is centered at the origin, and such that any polytope $P$
satisfying 
\[
K\subset P\subset RK,
\]
has at least $\exp(C\log(\frac{R^{2}}{n})\frac{n}{R^{2}}n)$ facets,
where $c_{0},\,c_{1},\,C>0$ are some universal constants. 
\end{thm}

${\bf Remark:}$ 
\begin{enumerate}
\item Notice that the inclusion relations are invariant under invertible
linear transformations. We may assume that the body $K$ is in John's
position. 
\item For each $R\in\left[c_{0}\sqrt{n},\,c_{1}n\right]$, the body $K$
which is constructed in the Theorem \ref{thm:MAIN} is a polytope.
Moreover, for any polytope $P$ satisfying $K\subset P\subset RK$,
we have: 
\[
\text{number of facets of P}\ge\frac{c}{n}\cdot\left(\text{number of facets of }K\right).
\]
\end{enumerate}
A direct consequence of Theorem \ref{thm:MAIN} answers Problem \ref{prob:main}
in the negative: 
\begin{cor}
Let $R_{n}\rightarrow+\infty$ be a positive increasing sequence that
satisfies $\lim_{n\rightarrow\infty}\frac{R_{n}}{n}\rightarrow0$.
For any constant $k>0$, there exists a convex body $K\subset\mathbb{R}^{n}$
in John's position for a sufficiently large $n$ such that there is
no polytope that has at most $n^{k}$ facets and satisfies 
\[
K\subset P\subset R_{n}K.
\]
\end{cor}

In the other extreme, we have the following corollary: 
\begin{cor}
\label{cor:sqrtn_scale}For a sufficiently large $n$, there exists
a convex body $K\subset\mathbb{R}^{n}$ in John's position such that
there is no polytope $P$ that has less than $\exp(cn)$ number of
facets and satisfies 
\[
K\subset P\subset\sqrt{n}K,
\]
where $c>0$ is a universal constant. 
\end{cor}

As we previously mentioned, for a symmetric convex body $K$, there
exists a polytope $P$ with $O\left(n^{2}\right)$ facets such that
$K\subset P\subset\sqrt{n}K$. Corollary \ref{cor:sqrtn_scale} shows
that approximating a non-symmetric body, in the same scale of $\sqrt{n}$
could be much more expensive.

The fact that Theorem \ref{thm:MAIN} cannot provide a better result
when $R=o(\sqrt{n})$ is not surprising. Using a net argument, one
can derive the following: 
\begin{prop}
\label{prop:NetBound}Suppose $B_{2}^{n}\subset K\subset RB_{2}^{n}$.
For a sufficiently small $\delta>0$, there exists a polytope $P_{\delta}$
with no more than $\exp(c\log(\frac{2R}{\delta})n)$ facets such that
\[
(1-\delta)P_{\delta}\subset K\subset P_{\delta}.
\]
\end{prop}

Applying the proposition to convex bodies in John's position, we conclude
the following. 
\begin{cor}
Let $K$ be a convex body in $\mathbb{R}^{n}$ in John's position,
where $n$ is sufficiently large. Then, there exists a polytope $P$
with at most $\exp(c\log(n)n)$ facets such that 
\[
\frac{1}{2}P\subset K\subset P,
\]
where $c>0$ is a universal constant.
\end{cor}

The paper is organized as follows. In Section \ref{sec:Preliminaries}
we introduce notation and recall some needed results. The proof of
Theorem \ref{thm:MAIN} is presented in Section \ref{sec:main}. In
Section \ref{sec:UpperBound} we prove Propostion \ref{prop:NetBound}. 

\section{Preliminaries \label{sec:Preliminaries}}

For the standard Euclidean space $\mathbb{R}^{n}$, let $\langle\cdot,\cdot\rangle$
denote the usual inner product. For a vector $x\in\mathbb{R}^{n}$,
let $|x|$ denote its Euclidean norm. Let $S^{n-1}$ be the unit sphere
and $B_{2}^{n}$ be the unit Euclidean ball. Let $GL_{n}\left(\R\right)$
denote the group of invertible linear transformations. For $D_{1},\dots,D_{d}\subset\R^{n}$,
let $\conv{\left\{ D_{1},\cdots,D_{d}\right\} }$ denote the convex
hull of the sets $D_{1},\dots,D_{d}$. Given a convex body $K\sub\R^{n},$
let $\partial K$ denote its boundary and $\vol K$ denote the Lebesgue
measure of $K$. We recall two standard functions associate with a
convex body: Suppose $K$ is a convex body containing $0$. The support
function $h_{K}$ of $K$ is defined by 
\[
h_{K}\left(x\right):=\sup\left\{ \iprod xy\,:\,y\in K\right\} \qquad\forall x\in\R^{n}.
\]

The radial function of $K$ is defined by 
\[
\rho_{K}\left(x\right):=\sup\left\{ r>0\,:\,rx\in K\right\} \qquad\forall x\in\R^{n}\backslash\left\{ 0\right\} .
\]

In particular, $\frac{1}{\rho_{K}\left(x\right)}$ is also called
the gauge function of $K$.

For a convex set $K\subset\mathbb{R}^{n}$ that contains $0$, we
define its polar $K^{\circ}$ by 
\[
K^{\circ}:=\{y\in\mathbb{R}^{n}\,:\,\forall x\in K\,\,\langle x,y\rangle\le1\}.
\]

Suppose $K_{1},\dots,K_{d}\subset\R^{n}$ are convex sets containing
the origin. One can verify that $\left(\cap_{i=1}^{d}K_{i}\right)^{\circ}=\conv{\left\{ K_{1}^{\circ},\cdots,K_{d}^{\circ}\right\} }$.

An ellipsoid $E\subset K$ is called the John ellipsoid of $K$ if
$\text{vol}(E)\ge\text{vol}(E')$ for any other ellipsoid $E'\subset K$.
It is well known that the John ellipsoid of a convex body exists and
is unique. Furthermore, for any convex body $K$, there exists an
affine transformation $T$ such that the John ellipsoid of $TK$ is
$B_{2}^{n}$. As we mentioned before, we say that $K$ is in John's
position if its John ellipsoid is $B_{2}^{n}$.

Let $K\subset\mathbb{R}^{n}$ be a convex body in John's position.
A point $x\in\mathbb{R}^{n}$ is said to be a contact point of $K$
and $B_{2}^{n}$ if $x\in\partial K\cap\partial B_{2}^{n}$. A classical
theorem of F. John provides a decomposition of the identity in terms
of contact points (see e.g., Citation \cite[p 52]{AGA}):

\begin{thm}
\label{JP} Let $K$ be a convex body in $\mathbb{R}^{n}$ that contains
$B_{2}^{n}$. Then, $K$ is in John's position if and only if there
exist contact points $x_{1},..,x_{m}$ and $a_{1},..,\,a_{m}>0$ such
that 
\end{thm}

\begin{enumerate}
\item $\sum_{i=1}^{m}a_{i}x_{i}\otimes x_{i}=I_{n}$, and 
\item $\sum_{i=1}^{m}a_{i}x_{i}=\vec{0}$. 
\end{enumerate}
Let $\Delta_{n}$ be the regular simplex in $\mathbb{R}^{n}$ that
has an inner radius equal to $1$. Using the symmetry of $\Delta_{n}$
and uniqueness of the John ellipsoid, it is not difficult to check
that $\Delta_{n}$ is in John's position. Suppose $u_{1},\cdots,u_{n+1}$
are the contact points of $\Delta_{n}$. Then, $\langle u_{i},u_{j}\rangle=-\frac{1}{n}$
for $i\neq j$ and $\left\{ -nu_{i}\right\} _{i=1}^{n+1}$ are the
vertices of $\Delta_{n}$. We can express $\Delta_{n}$ and $\Delta_{n}^{\circ}$
in terms of $\left\{ u_{i}\right\} _{i=1}^{n+1}$: 
\[
\Delta_{n}=\left\{ x\in\R^{n}\,:\,\iprod xu\le1\:\forall u\in\left\{ u_{i}\right\} _{i=1}^{n+1}\right\} ,\qquad\Delta_{n}^{\circ}=\conv{\left\{ u_{i}\right\} _{i=1}^{n+1}}.
\]

For any integer $m\in\mathbb{N}$, we define $[m]:=\{1,\cdots,m\}$.
For a subset $I\subset[m]$, let $|I|$ denote its cardinality. For
$i,j\in\mathbb{N}$, we define 
\begin{eqnarray*}
\delta_{ij}:=\left\{ \begin{array}{cc}
1 & \text{if \ensuremath{i=j},}\\
0 & \text{otherwise.}
\end{array}\right.
\end{eqnarray*}


\section{Proof of the main result\label{sec:main}}

In this section, we prove Theorem \ref{thm:MAIN}. The proof will
be divided into three main propositions. The body $K$ is obtained
by intersecting a simplex in John's position with a large number of
half spaces. As long as each half space contains the John ellipsoid
of the simplex, the new body will be in John's position as well. The
construction of the body uses both certain structures and randomness.

\subsection{Lower bound of facets}

The first proposition below shows how to determine if a convex body
cannot be approximated by polytopes with few facets.

\begin{prop}
\label{IdeaOfCounterExample} Let $K:=\{x\in\mathbb{R}^{n}\,:\,\langle x,y_{i}\rangle\le1\quad\forall i\in[m]\}\cap L$,
where $y_{1},...,y_{m}$ are vectors in $\mathbb{R}^{n}$ and $L$
is a convex body in $\mathbb{R}^{n}$ that has $0$ as an interior
point. Suppose there are points $x_{1},...,x_{m}\in K$ such that
for some $R>1$, we have 
\[
\langle x_{i},y\rangle\left\{ \begin{array}{cc}
=1 & \text{ if \ensuremath{y=y_{i}}, }\\
\le\frac{1}{2R} & \text{ if \ensuremath{y=y_{j}} with \ensuremath{i\neq j}.}\\
\le\frac{1}{2R} & \text{ if \ensuremath{y\in L^{\circ}}}
\end{array}\right.
\]
Then, there is no polytope $P$ that has less than $\frac{m}{2R}$
facets such that 
\[
K\subset P\subset RK.
\]
\end{prop}

\begin{proof}
Suppose there exists $\left\{ w_{i}\right\} _{i=1}^{m_{1}}\subset\mathbb{R}^{n}$
such that $P:=\{x\in\mathbb{R}^{n}\,:\,\langle x,w_{l}\rangle\le1\,\,\forall l\in[m_{1}]\}$
satisfies 
\[
K\subset P\subset RK.
\]

The first inclusion indicates that $\left\{ w_{i}\right\} _{i=1}^{m_{1}}\subset K^{\circ}$.
The second inclusion is equivalent to the following: $\forall x\in\partial K$,
$R\langle x,w_{l}\rangle\ge1$ for some $l\in[m_{1}]$. Due to $\iprod{x_{i}}{y_{i}}=1$,
we also have $x_{i}\in\partial K$ for $i\in\left[m\right]$.

For $l\in[m_{1}]$, let $O_{l}$ be the sub-collection of $\{x_{i}\}_{i=1}^{m}$
such that $R\langle x_{i},w_{l}\rangle\ge1$. Observe that $K^{\circ}=\text{conv}(\{y_{i}\}_{i=1}^{m},L^{\circ})$.
Thus, $w_{l}$ can be expressed as a convex combination: 
\[
w_{l}=\sum_{i=1}^{m}\lambda_{i}y_{i}+\lambda_{m+1}y,
\]

where $y\in L^{\circ}$, $\lambda_{i}\ge0$, and $\sum_{i=1}^{m+1}\lambda_{i}=1$.

This expression is not necessarily unique, but we fix one such expression.
Taking inner product with $Rx_{i}$ we have 
\begin{eqnarray*}
R\langle x_{i},w_{l}\rangle & = & \sum_{i\neq j}^{m}\lambda_{j}R\iprod{x_{i}}{y_{j}}+\lambda_{m+1}R\iprod{x_{i}}y+\lambda_{i}R\\
 & \le & \frac{1}{2}+\lambda_{i}R.
\end{eqnarray*}

If $x_{i}\in O_{l}$, then $\lambda_{i}\ge\frac{1}{2R}$. Due to $\sum_{i=1}^{m+1}\lambda_{i}=1$,
we conclude that $|O_{l}|\le2R$.

Observe that $\cup_{l\in[m_{1}]}O_{l}=\{x_{i}\}_{i\in[m]}$; we conclude
that $m_{1}\ge\frac{m}{2R}$. Therefore, $P$ has at least $\frac{m}{2R}$
facets. 
\end{proof}
The example in the main theorem will be of the form $K:=\{x\,:\,\langle x,y_{i}\rangle\le1\,\,\forall i\in[m]\}\cap\Delta_{n}$,
where $\{y_{i}\}_{i=1}^{m}\subset S^{n-1}$ and $\Delta_{n}$ is a
regular simplex in John's position. Then, we will find $\{x_{i}\}_{i=1}^{m}$,
which satisfies the assumption of Proposition \ref{IdeaOfCounterExample}.


\subsection{Structure}

Here we have a deterministic proposition related to points in $S^{n-1}$. 
\begin{prop}
\label{Structure}Let $S:=S^{n-1}\cap\{x\,:\,\langle\beta,x\rangle=0\}$
for some $\beta\in S^{n-1}$. For $\theta\in S$, let $\theta^{\downarrow}:=-\frac{1}{8}\beta+\sqrt{1-\left(\frac{1}{8}\right)^{2}}\theta$
and $\theta^{\uparrow}:=\sqrt{1-(\frac{1}{7})^{2}}\beta+\frac{1}{7}\theta$.
Then, 
\end{prop}

\begin{enumerate}
\item For $\alpha,\theta\in S$, $\langle\alpha^{\downarrow},\theta^{\uparrow}\rangle>0$
implies $\langle\alpha,\theta\rangle>\frac{3}{4}$, and 
\item $\langle\theta^{\uparrow},\theta^{\downarrow}\rangle=\frac{1}{C_{0}}$
for $\theta\in S$ where $C_{0}:=\frac{1}{\frac{1}{7}\sqrt{1-\left(\frac{1}{8}\right)^{2}}\left(1-\sqrt{\frac{48}{63}}\right)}>1$. 
\end{enumerate}
In our construction of $K$, $y_{i}$ will be $\theta_{i}^{\uparrow}$
for some $\theta_{i}\in S$ and $x_{i}$ will be $C_{0}\theta_{i}^{\downarrow}$.
In particular, the first statement of Proposition \ref{Structure}
implies that $\langle x_{i},y_{j}\rangle<0$ when $\iprod{\theta_{i}}{\theta_{j}}<\frac{3}{4}$. 
\begin{proof}
\begin{figure}[h]
\centering \includegraphics[scale=0.2]{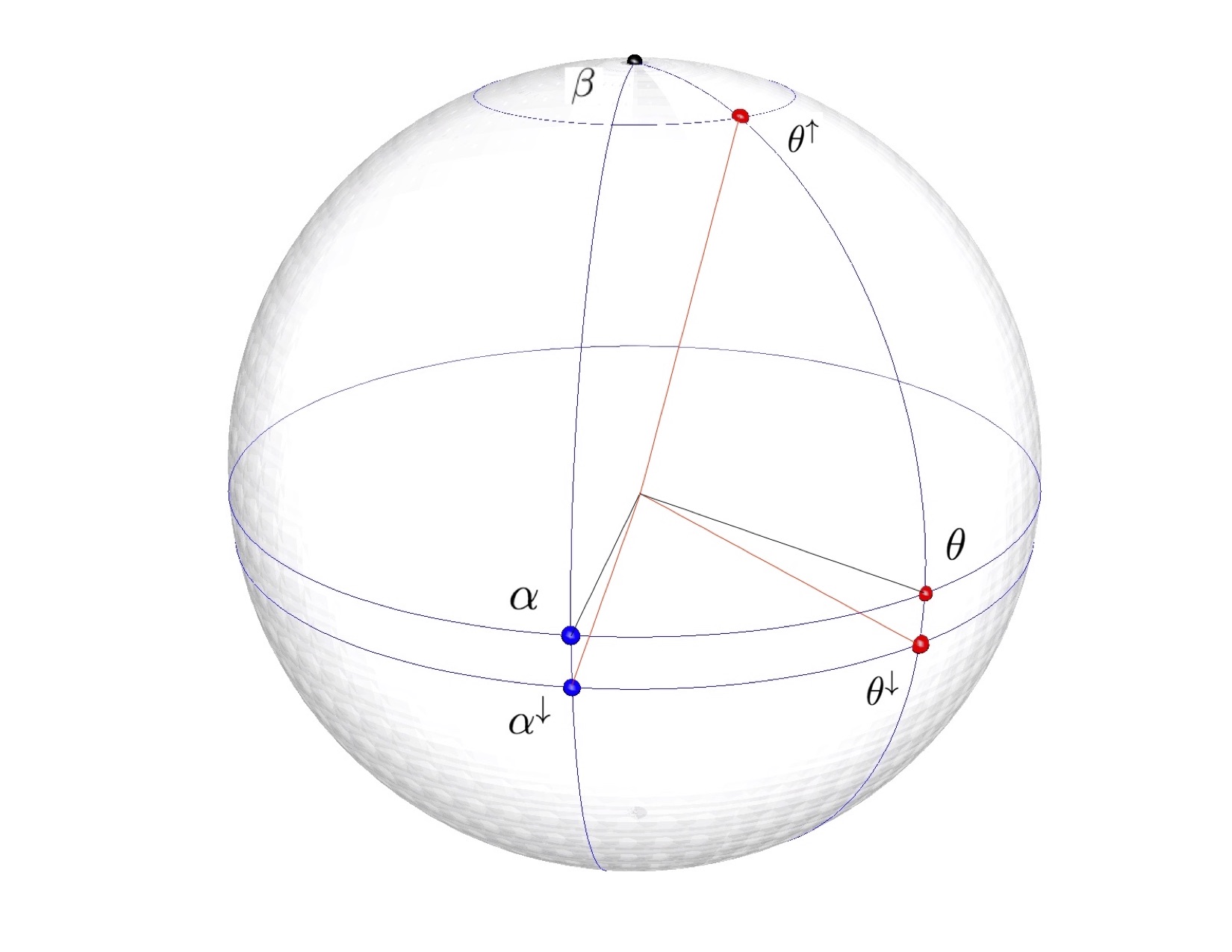} 
\end{figure}

We fix $\theta\in S$. For $\alpha\in S$, it can be expressed as
\begin{align*}
\alpha & =s\theta+\sqrt{1-s^{2}}\alpha',
\end{align*}

where $s=\iprod{\alpha}{\theta}$ and $\alpha'=\frac{\alpha-s\theta}{\left|\alpha-s\theta\right|}$.
Notice that $\alpha'\perp\theta$ and $\alpha'\in S$. Thus, 
\begin{eqnarray*}
\iprod{\alpha^{\downarrow}}{\theta^{\uparrow}} & = & \langle-\frac{1}{8}\beta+\sqrt{1-\left(\frac{1}{8}\right)^{2}}(s\theta+\sqrt{1-s^{2}}\alpha'),\sqrt{1-(\frac{1}{7})^{2}}\beta+\frac{1}{7}\theta\rangle\\
 & = & -\frac{1}{8}\sqrt{1-(\frac{1}{7})^{2}}+s\frac{1}{7}\sqrt{1-\left(\frac{1}{8}\right)^{2}}.
\end{eqnarray*}
Suppose $\iprod{\alpha^{\downarrow}}{\theta^{\uparrow}}>0$. Then,
$s\ge\frac{7}{8}\sqrt{\frac{48}{49}\frac{64}{63}}=\sqrt{\frac{48}{63}}>\frac{3}{4}.$

Observe that $\langle\theta^{\downarrow},\theta^{\uparrow}\rangle$
is the same for all $\theta\in S$, 
\begin{align*}
\langle\theta^{\downarrow},\theta^{\uparrow}\rangle & =\frac{1}{7}\sqrt{1-\left(\frac{1}{8}\right)^{2}}-\frac{1}{8}\sqrt{1-(\frac{1}{7})^{2}}\\
 & =\frac{1}{7}\sqrt{1-\left(\frac{1}{8}\right)^{2}}\left(1-\sqrt{\frac{48}{63}}\right).
\end{align*}

Since $0<\frac{1}{7}\sqrt{1-\left(\frac{1}{8}\right)^{2}}\left(1-\sqrt{\frac{48}{63}}\right)<1$,
we conclude that $C_{0}=\frac{1}{\frac{1}{7}\sqrt{1-\left(\frac{1}{8}\right)^{2}}\left(1-\sqrt{\frac{48}{63}}\right)}>1.$ 
\end{proof}

\subsection{Randomness}

In the construction, we will choose $\left\{ \theta_{i}\right\} $
independently and uniformly according to a probability distribution.
In order to apply Proposition \ref{IdeaOfCounterExample}, we need
to choose $\theta_{i}$ so that $\iprod{x_{i}}y\le\frac{1}{2R}$ for
all $y\in\Delta_{n}^{\circ}$. Equivalently, $\rho_{\Delta_{n}}(\theta_{i}^{\downarrow})$
needs to be larger than $2RC_{0}$. The uniform randomness on $S$
does not work in this case. Thus, a probability that is compatible
with the structure of $\Delta_{n}$ is required.

The following proposition is a tail bound for the hypergeometric distribution. 
\begin{prop}
\label{Tailbound} For a sufficiently large $n\in\mathbb{N}_{+}$,
let $k$ be a positive integer satisfying $100<k<\frac{1}{2e^{8}}n$.
Suppose $I,\,J$ are chosen independently and uniformly from $\left\{ W\subset\left[n\right]\,:\,\left|W\right|=k\right\} $.
Then, 
\[
\PP\left(\left|I\cap J\right|\ge\frac{k}{2}\right)\le\left(\frac{2k}{n}\right)^{k/5}.
\]
\end{prop}

\begin{proof}
We may assume that $J$ is fixed. For any positive integer $1\le l\le k$,
\begin{equation}
\mathbb{P}(|I\cap J|=l)=\frac{{k \choose l}{n-k \choose k-l}}{{n \choose k}}.\label{RandomProb1}
\end{equation}
For positive integers $a\ge b$, ${a \choose b}=\frac{a(a-1)\cdots(a-b+1)}{b(b-1)\cdots1}$.
A standard estimate of ${a \choose b}$ is the following: 
\[
(\frac{a}{b})^{b}\le{a \choose b}\le(\frac{ea}{b})^{b}.
\]
Applying these bounds to (\ref{RandomProb1}), we have 
\begin{eqnarray}
\begin{split}\mathbb{P}(|I\cap J|=l) & \le & (\frac{ek}{l})^{l}(\frac{e(n-k)}{k-l})^{k-l}(\frac{k}{n})^{k}\\
 & = & e^{k}(\frac{k^{2}}{ln})^{l}(\frac{n-k}{n})^{k-l}(\frac{k}{k-l})^{k-l}.
\end{split}
\label{RandomProb2}
\end{eqnarray}
Assuming that $l\ge\frac{k}{2}$ and $2k<n$, 
\[
(\frac{k^{2}}{ln})^{l}\le(\frac{2k}{n})^{l}\le(\frac{2k}{n})^{\frac{k}{2}}.
\]
Also, using $(1+x)\le e^{x}$ for $x\in\mathbb{R}$ we have 
\[
(\frac{k}{k-l})^{k-l}=(1+\frac{l}{k-l})^{k-l}\le e^{l}\le e^{k}.
\]
Together with $(\frac{n-k}{n})^{k-l}\le1$, 
\begin{eqnarray*}
\mathbb{P}(|I\cap J|=l) & \le & \exp(k-\log(\frac{n}{2k})\frac{k}{2}+k)\\
 & = & \exp(2k-\log(\frac{n}{2k}))\frac{k}{2}).
\end{eqnarray*}
If $\frac{n}{2k}\ge e^{8}$, then $2k\le\log(\frac{n}{2k})\frac{k}{4}$.
We have the following: 
\[
\mathbb{P}(|I\cap J|=l)\le\exp(-\log(\frac{n}{2k})\frac{k}{4})\qquad\forall l\ge\frac{k}{2}.
\]
Using the union bound we get 
\[
\mathbb{P}(|I\cap J|\ge\frac{k}{2})\le k\exp(-\log(\frac{n}{2k})\frac{k}{4})\le\exp(-\log(\frac{n}{2k})\frac{k}{5})=(\frac{2k}{n})^{k/5},
\]

where the last inequality requires $k\ge100$. 
\end{proof}

\subsection{The construction}

Let $\Delta_{n}$ be the $n-$dimensional simplex in John's position
and $u_{1},\cdots,u_{n+1}$ be its contact points. We define $S:=S^{n-1}\cap\{x\in\mathbb{R}^{n}\,:\,\langle x,u_{1}\rangle=0\}$.
Notice that $\Delta_{n}'=\Delta_{n}\cap\{\langle x,u_{1}\rangle=0\}$
is a $\left(n-1\right)$-dimensional regular simplex. Let $v_{1},...,v_{n}\in S$
such that $\left\{ c_{n}nv_{i}\right\} _{i=1}^{n}$ are vertices of
$\Delta_{n}'$. It is not hard to verify that $c_{n}\rightarrow1$
as $n\rightarrow\infty$. Let $1\le k\le n$. For $I\subset[n]$ with
$|I|=k$, we define 
\[
v_{I}=\frac{\sum_{i\in I}v_{i}}{|\sum_{i\in I_{j}}v_{i}|}\in S.
\]
\begin{prop}
\label{innerproduct} For a sufficiently large $n$, let $1\le k\le\frac{n+3}{4}$.
Suppose $I,J\subset\left\{ W\subset\left[n\right]\,,\,\left|W\right|=k\right\} $
satisfy $|I\cap J|<\frac{k}{2}$. Then, 
\[
\langle v_{I},v_{J}\rangle\le\frac{3}{4}.
\]
\end{prop}

\begin{proof}
Because $\left\{ v_{i}\right\} _{i=1}^{n}$ are in the vertex directions
of a regular simplex, $\langle v_{i},v_{j}\rangle=-\frac{1}{n-1}$
if $i\neq j$. We have 
\begin{eqnarray*}
\langle\sum_{i\in I}v_{i},\sum_{i\in I}v_{i}\rangle & = & \sum_{i,j\in I}\left(-\frac{1}{n-1}+\delta_{ij}(1+\frac{1}{n-1})\right)\\
 & = & -\frac{k^{2}}{n-1}+k(1+\frac{1}{n-1})\\
 & = & k(1-\frac{k-1}{n-1}).
\end{eqnarray*}

Thus, 
\[
v_{I}=\frac{c_{n,k}}{\sqrt{k}}\sum_{i\in I}v_{i},
\]
where $c_{n,k}=\frac{1}{\sqrt{1-\frac{k-1}{n-1}}}$. Suppose $1\le k\le\frac{n+3}{4}$,
then $1\le c_{n,k}\le\sqrt{\frac{4}{3}}$.

Let $J\subset[n]$ with $|J|=k$. Then, 
\begin{eqnarray*}
\frac{k}{c_{n,k}^{2}}\langle v_{I},v_{J}\rangle & = & \sum_{i\in I}\sum_{j\in J}\langle v_{i},v_{j}\rangle\\
 & = & \sum_{i\in I}\sum_{j\in J}\left(-\frac{1}{n-1}+(1+\frac{1}{n-1})\delta_{ij}\right)\\
 & = & -\frac{k^{2}}{n-1}+|I\cap J|(1+\frac{1}{n-1})\\
 & \le & |I\cap J|(1+\frac{1}{n-1}).
\end{eqnarray*}
Suppose $|I\cap J|<\frac{k}{2}$ and $n$ is large enough. Then, 
\[
\langle v_{I},v_{J}\rangle\le\frac{c_{n,k}^{2}}{2}(1+\frac{1}{n-1})\le\frac{2}{3}(1+\frac{1}{n-1})<\frac{3}{4}.
\]
\end{proof}
We are now ready to prove the main theorem. 
\begin{proof}[Proof of Theorem \ref{thm:MAIN} ]
For a sufficiently large $n$, we fix $100\le k\le\frac{n}{2e^{8}}$.
Then, $n$ and $k$ satisfy the assumptions in Proposition \ref{Tailbound}
and Proposition \ref{innerproduct}. Let $m\subset\mathbb{N}$ be
an integer that we will specify later. Let $I_{1},\,I_{2},\dots,\,I_{m}$
be chosen independently and uniformly from $\{W\subset[n]\,:\,|W|=k\}$.
Let $u_{1}$ be the vector $\beta$ described in Proposition \ref{Structure}.
In particular, $S:=S^{n-1}\cap\{x\,:\,\iprod{u_{1}}x=0\}$. We adapt
the definition of $\theta^{\uparrow}$ and $\theta^{\downarrow}$
for $\theta\in S$. Let 
\[
K:=\Delta_{n}\cap(\cap_{i=1}^{m}\{x\in\mathbb{R}^{n}\,:\,\langle x,u_{I_{i}}^{\downarrow}\rangle\le1\}).
\]
By Proposition \ref{Tailbound}, we have $\mathbb{P}(|I_{i}\cap I_{j}|\ge\frac{k}{2})\le(\frac{2k}{n})^{\frac{k}{5}}$.
A union bound argument shows that 
\begin{eqnarray*}
\mathbb{P}(\exists1\le i<j\le m\text{ such that }|I_{i}\cap I_{j}|\ge\frac{k}{2})\le{m \choose 2}(\frac{2k}{n})^{\frac{k}{5}}<m^{2}(\frac{2k}{n})^{\frac{k}{5}}.
\end{eqnarray*}
By setting $m=(\frac{n}{2k})^{k/20}$, we have 
\begin{equation}
\mathbb{P}(\exists1\le i<j\le m\text{ such that }|I_{i}\cap I_{j}|\ge\frac{k}{2})\le(\frac{2k}{n})^{\frac{k}{10}}.
\end{equation}
Since the probability is strictly smaller than $1$, there exists
a sample such that $|I_{i}\cap I_{j}|<\frac{k}{2}$ for all $1\le i<j\le m$.
From now on, we fix such a sample.

We want to apply Proposition \ref{IdeaOfCounterExample} with $L=\Delta_{n}$,
$y_{i}=v_{I_{i}}^{\uparrow}$, and $x_{i}=\frac{1}{\langle v_{I_{i}}^{\downarrow},\,v_{I_{i}}^{\uparrow}\rangle}v_{I_{i}}^{\downarrow}=C_{0}v_{I_{i}}^{\downarrow}$,
where $C_{0}$ is the constant defined in Proposition \ref{Structure}.
We start verifying the assumptions that are described in Proposition
\ref{IdeaOfCounterExample}.

First, $\Delta_{n}^{\circ}=\text{conv}\{u_{1},\cdots,u_{n+1}\}$.
Because $\{c_{n}nv_{i}\}_{i=1}^{m}\subset\Delta_{n}^{'}\subset\Delta_{n}$,
$\langle c_{n}nv_{i},u_{j}\rangle\le1$ for $i\in[n]$ and $j\in[n+1]$.
Thus, for any $I\subset[n]$ with $|I|=k$ and $j\in[n+1]$, 
\[
\langle v_{I},u_{j}\rangle=\frac{c_{n,k}}{\sqrt{k}}\sum_{i\in I}\langle v_{i},u_{j}\rangle\le\frac{c_{n,k}}{\sqrt{k}}k\frac{1}{c_{n}n}\le\frac{c_{n,k}\sqrt{k}}{c_{n}n}.
\]
Since $\langle-u_{i},u_{j}\rangle=-(1+\frac{1}{n})\delta_{ij}+\frac{1}{n}\le\frac{1}{n}$,

\begin{eqnarray*}
\langle v_{I}^{\downarrow},u_{j}\rangle & = & \frac{1}{8}\langle-u_{1},u_{j}\rangle+\sqrt{1-\left(\frac{1}{8}\right)^{2}}\langle v_{I},u_{j}\rangle\\
 & \le & \frac{1}{8}\frac{1}{n}+\sqrt{1-\left(\frac{1}{8}\right)^{2}}\frac{c_{n,k}\sqrt{k}}{c_{n}n}.
\end{eqnarray*}
Since $\frac{n}{2e^{8}}\ge k\ge1$ , $1\le c_{n,k}\le\sqrt{\frac{4}{3}}$
and $c_{n}\rightarrow1$ as $n\rightarrow+\infty$, 
\begin{align*}
\frac{1}{8}\frac{1}{n}+\sqrt{1-\left(\frac{1}{8}\right)^{2}}\frac{c_{n,k}\sqrt{k}}{c_{n}n} & \le\frac{1}{n}+\frac{1}{c_{n}}\sqrt{\frac{4}{3}}\frac{\sqrt{k}}{n}\\
 & \le\left(\frac{1}{c_{n}}\sqrt{\frac{4}{3}}+1\right)\frac{\sqrt{k}}{n}\\
 & \le3\frac{\sqrt{k}}{n}.
\end{align*}

We conclude that

\begin{align*}
\langle C_{0}v_{I}^{\downarrow},u_{j}\rangle & \le3C_{0}\frac{\sqrt{k}}{n}.
\end{align*}
Every $y\in\Delta_{n}^{\circ}$ can be written as a convex combination
of $\left\{ u_{i}\right\} _{i=1}^{n+1}$. Thus, the same inequality
holds: 
\begin{equation}
\langle C_{0}v_{I_{i}}^{\downarrow},\,y\rangle\le3C_{0}\frac{\sqrt{k}}{n}\qquad\forall y\in\Delta_{n}^{\circ}.
\end{equation}

Let $i,j\in[m+1]$ with $i\neq j$. Applying Proposition \ref{innerproduct},
we obtain $\langle v_{I_{i}},v_{I_{j}}\rangle<\frac{3}{4}$ since
$\left|I_{i}\cap I_{j}\right|<\frac{3}{4}$. According to Proposition
\ref{Structure}, we obtain $\langle C_{0}v_{I_{i}}^{\downarrow},v_{I_{j}}^{\uparrow}\rangle<0$.
In the case $i=j$, by definition we have $\langle C_{0}v_{I_{i}}^{\downarrow},v_{I_{j}}^{\uparrow}\rangle=1$.
To summarize, 
\begin{eqnarray}
\langle c_{1}v_{I_{i}}^{\downarrow},y\rangle\left\{ \begin{array}{cc}
=1 & \text{ if \ensuremath{y=v_{I_{i}}^{\uparrow}}, }\\
\le0 & \text{ if \ensuremath{y=v_{I_{j}}^{\uparrow}} with \ensuremath{j\neq i},}\\
\le3C_{0}\frac{\sqrt{k}}{n} & \text{ if \ensuremath{y\in\Delta_{n}^{\circ}}. }
\end{array}\right.
\end{eqnarray}
Now, we can apply Proposition \ref{IdeaOfCounterExample} with $m=(\frac{n}{2k})^{k/10}$,
$y_{i}=u_{I_{i}}^{\uparrow}$, $x_{i}=C_{0}u_{I_{i}}^{\downarrow}$,
$L=\Delta_{n}$ and $R=\frac{n}{6C_{0}\sqrt{k}}$ with the condition
that $100\le k\le\frac{n}{2e^{8}}$. Expresssing these relations in
terms of $R$ and $n$, we have 
\[
k=(\frac{n}{6C_{0}R})^{2},\qquad m=\left(\frac{18C_{0}^{2}R^{2}}{n}\right)^{(\frac{n}{6C_{0}R})^{2}/20},\quad\text{and }\frac{\sqrt{2}e^{4}}{6C_{0}}\sqrt{n}\le R\le\frac{n}{60C_{0}}.
\]
The lower bound of the facets of the polytope $P$ in Proposition
\ref{IdeaOfCounterExample} is $\frac{m}{2R}$. To simplify $m$,
we further restrict $R>\sqrt{en}$ so that $\frac{R^{2}}{n}>e$. Since
$C_{0}>1$, we have $\log\left(18C_{0}^{2}\right)>0$. Thus, 
\[
\log(\frac{18C_{0}^{2}R^{2}}{n})=\log(18C_{0}^{2})+\log(\frac{R^{2}}{n})\ge\log(\frac{R^{2}}{n})>1>0.
\]
Then, 
\begin{eqnarray*}
\frac{m}{2R} & = & \exp(-\log(2R)+\log(\frac{18C_{0}^{2}R^{2}}{n})\frac{n^{2}}{720C_{0}^{2}R^{2}})\\
 & \ge & \exp(-\log(2R)+\frac{1}{720C_{0}^{2}}\log(\frac{R^{2}}{n})\frac{n^{2}}{R^{2}})\\
 & \ge & \exp(-\log(2n)+C'\log(\frac{R^{2}}{n})\frac{n^{2}}{R^{2}}),
\end{eqnarray*}
when $C'=\frac{1}{720C_{0}^{2}}>0$. In order to take care of the
$\log(2n)$ term we need to check the last term carefully. First,
\[
\frac{d}{dR}\log(\frac{R^{2}}{n})\frac{n^{2}}{R^{2}}=-\frac{2n^{2}}{R^{3}}\left(\log(\frac{R^{2}}{n})-1\right)<0
\]
for $R>\sqrt{en}$. Let $c_{1}=\min\left\{ 1,\,\sqrt{\frac{C'}{8}},\,\frac{1}{60C_{0}}\right\} $.
Suppose $R=c_{1}n$, we have 
\[
C'\log(\frac{R^{2}}{n})\frac{n^{2}}{R^{2}}=\frac{C'}{c_{1}^{2}}\log(n)+\frac{C'}{c_{1}^{2}}\log(c_{1}{}^{2})>4\log\left(n\right),
\]

where the last inequality holds for large $n$. Together with $2\log(n)\ge\log(2n)$,
\begin{align*}
\frac{1}{2}C'\log(\frac{R^{2}}{n})\frac{n^{2}}{R^{2}} & \ge\log(2n)
\end{align*}
when $R=c_{1}n$. Since $\log(\frac{R^{2}}{n})\frac{n}{R^{2}}$ is
a decreasing function for $R>\sqrt{en}$, $\frac{1}{2}C'\log(\frac{R^{2}}{n})\frac{n}{R^{2}}\ge\log(2n)$
for $\sqrt{en}<R<c_{1}n$. Therefore, we conclude that for $c_{0}\sqrt{n}<R<c_{1}n$,
\[
\frac{m}{2R}\ge\exp(C\log(\frac{R^{2}}{n})\frac{n^{2}}{R^{2}})
\]
where $C>0$ is an universal constants. Therefore, for $c_{0}\sqrt{n}\le R\le c_{1}n$,
there exists a convex body $K\subset\mathbb{R}^{n}$ in John's position
such that no polytope $P$ that has less than $\exp(C\log(\frac{R^{2}}{n})\frac{n^{2}}{R^{2}})$
facets satisfies 
\[
K\subset P\subset RK.
\]
\end{proof}

\section{Upper bound for small $R$ \label{sec:UpperBound}}
\begin{prop}
Suppose $B_{2}^{n}\subset K\subset RB_{2}^{n}$. For $0<\delta<1$,
there exists a polytope $P_{\delta}$ with no more than $\exp(c\log(\frac{2R}{\delta})n)$
facets such that $(1-\delta)P_{\delta}\subset K\subset P_{\delta}$.
Here $c>0$ is a universal constant. 
\end{prop}

\begin{proof}
Let $B_{2}^{n}\subset K\subset RB_{2}^{n}$ be a convex body. Let
$h:S^{n-1}\rightarrow\left[1,\,R\right]$ be the support function
of $K$. Observe that $h$ is also the gauge function of $\frac{1}{R}B_{2}^{n}\subset K^{\circ}\subset B_{2}^{n}$
. Thus, $h$ is a $R-$Lipschitz continuous function.

Let $\mathscr{N}$ be a $\frac{\delta}{2R}$-net of $S^{n-1}$. We
define 
\begin{align*}
P_{\delta}: & =\left\{ x\in\R^{n}\,:\,\forall\alpha\in\mathscr{N}\,\iprod{\alpha}x\le h\left(\alpha\right)\right\} .
\end{align*}

Thus, $P_{\delta}$ is a polytope with at most $\left|\mathscr{N}\right|$
facets. Recall that by a volumetric argument, the size of a $\eps-$net
on $S^{n-1}$ is bounded by $\exp\left(c\log\left(\frac{1}{\eps}\right)n\right)$
for an universal constant $c>0$. Hence, $P_{\delta}$ has no more
than $\exp\left(c\log\left(\frac{2R}{\delta}\right)n\right)$ number
of facets.

Since 
\[
K=\left\{ x\in\R^{n}\,:\,\forall\alpha\in S^{n-1}\,\iprod{\alpha}x\le h\left(\alpha\right)\right\} ,
\]
we have $K\subset P_{\delta}$. Observe that 
\[
\left(1-\delta\right)P_{\delta}=\left\{ x\in\R^{n}\,:\,\forall\alpha\in\mathscr{N}\,\iprod{\alpha}x\le\left(1-\delta\right)h\left(\alpha\right)\right\} .
\]

For $x\in\partial K$, there exists $\theta$ such that $\iprod x{\theta}=h\text{\ensuremath{\left(\theta\right)}}$.
We pick $\alpha\in\mathscr{N}$ such that $\norm{\alpha-\theta}<\frac{\delta}{2R}$.
Then, 

\begin{align}
\iprod x{\alpha} & =\iprod x{\theta}+\iprod x{\alpha-\theta}\ge h\left(\theta\right)-\left|x\right|\left|\alpha-\theta\right|\nonumber \\
 & >h\left(\theta\right)-R\frac{\delta}{2R}=h\left(\theta\right)-\frac{\delta}{2}.\label{eq:innerproduct}
\end{align}

Since $h$ is a $R-$Lipschitz continuous function, we have 
\[
h\left(\theta\right)\ge h\left(\alpha\right)-\frac{\delta}{2}.
\]

Together with $h\left(\alpha\right)\ge1$, the equation (\ref{eq:innerproduct})
becomes 
\begin{align*}
\iprod x{\alpha} & >h\left(\alpha\right)-\delta\ge\left(1-\delta\right)h\left(\alpha\right).
\end{align*}

Thus, $x\notin\left(1-\delta\right)P_{\delta}$. In paricular, we
conclude the radial function of $K$ is always greater than the radial
funciton of $\left(1-\delta\right)P_{\delta}$. Therefore, we have
$\left(1-\delta\right)P_{\delta}\subset K$. 
\end{proof}

\section*{Acknowledgement}

I am grateful to my Advisor, Mark Rudelson for fruitful discussions.
I am also grateful to the anonymous reviewer for his/her suggestions.

\end{document}